\documentclass[11pt,draft, leqno]{article}

\usepackage{amsmath}
\usepackage{color}
\usepackage{indentfirst}
\usepackage{verbatim}   
\usepackage{mathrsfs}   
\usepackage{yhmath}     

\usepackage{comment}
\usepackage{latexsym,amsfonts,amssymb, amscd}
\usepackage{bbm}
\usepackage{pifont}

\newtheorem{theorem}{Theorem}[section]
\newtheorem{lemma}[theorem]{Lemma}
\newtheorem{corollary}[theorem]{Corollary}
\newtheorem{prop-def}{Proposition-Definition}[section]

\newtheorem{definition}[theorem]{Definition}

\newenvironment{proof}{\trivlist \item[\hskip \labelsep{\it Proof.}]}{\endtrivlist}

\begin{document}


\baselineskip=15pt


\title{The $(\mu, \rho, \beta)$-Extension of $3$-Lie algebras}
\author{Ruipu Bai, Yansha Gao, Zhenheng Li}
\date{}
\maketitle


\renewcommand{\thefootnote}{}

\footnote{2010 \emph{Mathematics Subject Classification}: Primary 17B05; Secondary 17D99.}

\footnote{\emph{Key words and phrases}: $3$-Lie algebra, extension, derivation, tensor $3$-Lie algebra.}

\renewcommand{\thefootnote}{\arabic{footnote}}
\setcounter{footnote}{0}

\vspace{-2cm}

\begin{abstract}
We study an extension algebra $A$ from two given $3$-Lie algebras $M$ and $H$, and discuss the extensibility of a pair of derivations, one from the derivation algebra of $M$ and the other from that of $H$, to a derivation of $A$. In particular, we give conditions for such an extension to be a $3$-Lie algebra, and provide necessary and sufficient conditions of the pair of derivations to be extendable.

\end{abstract}

\section{Introduction}

Construction of $3$-Lie algebras is challenging  in studying the structure of $3$-Lie algebras.
Since the concept of  $n$-Lie algebra \cite{F}  was introduced in 1985, how to get $3$-Lie algebras from known algebraic structures
 has been  continually investigated. We give a short summary of different kinds of 3-Lie algebras constructed from different application areas.

In 1973, motivated by some problems of quark dynamics, authors in \cite{N, T} introduced
a $3$-ary generalization of Hamiltonian dynamics by means of the $3$-ary Poisson bracket,
$$
[f_1, f_2, f_3]= \det(\frac{\partial f_i}{\partial x_j}).
$$
This is a special case of the $3$-Lie algebra given in  \cite{F, Po1, Po2}.  Let $(A,\cdot)$ be  a commutative associative algebra
and  $\{D_1, D_2, D_3\}$ be three pairwise commuting derivations.
Then there is a $3$-Lie algebra structure on $A$ which is called a
{\it Jcobian algebra} defined by $$[x_{1}, x_{2},
x_3]=\det \left(
                                        \begin{array}{ccc}
                                          D_1(x_1) &D_1(x_2) &D_1(x_3)\\
                                          D_2(x_1) &D_2(x_2) &D_2(x_3)\\
                                          D_3(x_1) &D_3(x_2) &D_3(x_3)\\
                                        \end{array}
                                      \right).$$

 In papers \cite{G, STS}, $3$-Lie algebras are constructed by Dirac $\gamma$-matrices.
 Let $A$ be a vector space
spanned by the four-dimensional $\gamma$-matrices $(\gamma^\mu)$
and let $\gamma^5=\gamma^1\dots\gamma^4$. Then the product
$$
[x,y,z]=[[x,y]\gamma^5,z],\;\;\forall x, y, z\in A,
$$
defines a $3$-Lie algebra which is isomorphic to the unique simple
3-Lie algebra \cite{L}.

For studying integral systems, $3$-Lie algebras are constructed by $2$-dimen-sional extensions of metric Lie algebras  \cite{HIM}.
Let $(\frak g, B)$ be a metric Lie algebra over a field $\mathbb F$,
that is, $B$ is a nondegenerate symmetric bilinear form on $\frak g$
satisfying $B([x, y], z)=-B(y, [x, z])$ for every $x, y, z\in \frak
g$.
 Suppose $\{x_1, \cdots, x_m\}$ is a basis of $\frak g$
and $[x_i, x_j]=\sum\limits_{k=1}^m a_{ij}^kx_k, ~1\leq i, j\leq m.$
Set
$$\frak g_0=\frak g\oplus \mathbb F x^0\oplus
\mathbb F x^{-1}~ \mbox{ (a direct sum of vector
spaces)}.$$
  Then there is a 3-Lie algebra structure on
$\frak g_0$ given by
$$
[x_0, x_i, x_j]=[x_i, x_j], ~ 1\leq i, j\leq m; ~ [x^{-1}, x_i,
x_j]=0, ~ 0\leq i, j\leq m;$$
$$ [x_i, x_j,
x_k]=\sum\limits_{s=1}^ma_{ij}^sB(x_s, x_k)x^{-1}, ~1\leq i, j,
k\leq m.$$
The general linear Lie algebras with trace forms
\cite{ALMY} are constructed in the matrix space. Let $\frak g=gl(m,\mathbb F)$ be the general linear
Lie algebra. Then there is a 3-Lie algebra structure on $\frak g$
defined by
$$
[A, B, C] = (trA)[B, C] + (trB)[C, A] + (trC)[A, B], \;\;\forall A,
B, C\in \frak g.$$

R. Bai and cooperators constructed $3$-Lie algebras using  Lie algebras and linear functions in \cite{BBW, BLZ33}. Let $\frak g$ be a Lie algebra
and $f\in \frak g^*$ satisfying $f([\frak g, \frak g])=0$, then $\frak g$ is a $3$-Lie algebra in the multiplication
$$
[x, y, z]_f=f(x)[y, z]+f(y)[z, x]+f(z)[x, y],  ~~\forall x, y, z\in \frak g.
$$
It is proved that all non-simple  $m$-dimensional $3$-Lie algebras with $m\leq 5$ can be realized by Lie algebras. In paper \cite{BW22},  $3$-Lie algebras
are constructed from  a commutative associative algebras. Let $A$ be a commutative associative algebra with a derivation $\delta$ and an involution $\Delta$ satisfying $\delta\Delta+\Delta\delta=0$,
then $A$ is a $3$-Lie algebra in the multiplication $$[x, y, z]=(\Delta\wedge Id_A\wedge \delta)(x, y, z).$$
Using this method, we can easily get simple or non-simple infinite dimensional $3$-Lie algebras from group algebras.

In this paper, We study an extension algebra $A$ from two given $3$-Lie algebras $M$ and $H$, and discuss the extensibility of a pair of derivations, one from the derivation algebra of $M$ and the other from that of $H$, to a derivation of $A$. In particular, we give conditions for such an extension to be a $3$-Lie algebra, and provide necessary and sufficient conditions of the pair of derivations to be extendable.

 We suppose that $3$-Lie algebras are over a field $\mathbb F$ with $ch \mathbb F\neq 2.$

\section{Preliminary}

 A {\bf 3-Lie algebra} is a vector space $A$ endowed with a ternary multi-linear skew-symmetric operation
satisfying
for all $x,y, z, u, v\in A$.
\begin{equation}
[[x, y, z],u, v]=[[x, u, v],y, z] +[x, [y,u, v],z]+[x, y, [z, u, v]].
\end{equation}
\\
A {\bf derivation} of a $3$-algebra $A$ is a linear map $d: A\to A$ such that
\begin{equation}
d([x_1,x_2,x_3])= [ d(x_1),x_2,x_3] +[ x_1,d(x_2),x_3] +[ x_1,x_2,d(x_3)]
\end{equation}
 All the derivations of $A$, denoted by $Der(A)$, is a linear Lie algebra.

Let $A$ be a $3$-Lie algebra, $V$ be a vector space, and
$\rho: A\wedge A\rightarrow End (V)$ be a linear map. If $\rho$ satisfies that
\vspace{1mm}

{\small
$[\rho(x_1, x_2), \rho(x_3, x_4)]=\rho(x_1, x_2)\rho(x_3, x_4)-\rho(x_3, x_4)\rho(x_1, x_2)$
\\$=\rho([x_1, x_2, x_3], x_4)-\rho([x_1, x_2, x_4], x_3),$ and

\vspace{1mm}$\rho([x_1, x_2, x_3], x_4)=\rho(x_1, x_2)\rho(x_3, x_4)+\rho(x_2, x_3)\rho(x_1, x_4)+\rho(x_3, x_1)\rho(x_2, x_4),$
\\
}
then $(V, \rho)$ is called {\bf a representation} of $A$, or $(V, \rho)$ is {\bf an $A$-module}.

\begin{lemma} Let $A$ be a $3$-Lie algebra over $F$, and $V$ be a vector space,  and $\rho: A\wedge A\rightarrow End (V)$ be a linear map. If $(V, \rho)$ is an $A$-module, then the following identities hold, for all $x, y, z, u\in A,$

\begin{equation}
\rho([x, y, z], u)-\rho([x, y, u], z)+\rho([x, z, u], y)-\rho([y, z, u],x)=0,
\label{eq:rbne21}
\end{equation}
\begin{equation}
\rho(x,u)\rho(y, z)+\rho(y, z)\rho(x, u)+\rho(x, y)\rho(z, u)+\rho(z, u)\rho(x, y)
\label{eq:rbn22}
\end{equation}

\hspace{1cm}$-\rho(x, z)\rho(y, u)-\rho(y, u)\rho(x, z)=0.$
\label{lem:le21}
\end{lemma}

\begin {proof} The result follows from a direct computation.

\end{proof}

\section{$(\mu, \rho, \beta)$-extension of $3$-Lie algebras}

\begin{definition}
Let $(H, [ , , ]_H)$ and $( M, [ , , ]_M)$ be $3$-Lie algebras over a field $F$, $A=M\dot+ H$, and

$\rho: M\wedge M\rightarrow Der(H),$ ~  $ \beta: M\wedge H\rightarrow Der(H),$ ~  $ \mu: M\wedge M \wedge M\rightarrow H$ \\be linear mappings.
 Define    linear  multiplication $[ , , ]_{\mu\rho\beta}$: $A\wedge A\wedge A\rightarrow A$ by, for all $x, y, z\in M$ and $h, h_1, h_2\in H$,
\begin{equation}
  {[}x, y, z]_{\mu\rho\beta}=[x, y, z]_M+\mu(x, y, z),  {[}x, y, h]_{\mu\rho\beta} ~ =\rho(x, y)h,
  \label{eq:b1}
\end{equation}

\hspace{1.4cm}
$[h_1, h_2, h_3]_{\mu\rho\beta}=[h_1, h_2, h_3]_H, [x, h_1, h_2]_{\mu\rho\beta}~ =\beta(x, h_1)h_2.
 $
 \\ Then the $3$-algebra algebra $(A, [ , , ]_{\mu\rho\beta})$ is called a
{\bf $(\mu, \rho, \beta)$-extension} of $H$ by $M$.

\vspace{2mm}
If $\beta=0$, then $A$ is simply called a $(\mu, \rho)$-extension of $H$ by $M$, and $[ , , ]_{\mu\rho\beta}$ is reduced to $[ , ,  ]_{\mu\rho}$.

\end{definition}

For convenience, in the following,  $[ , , ]_H$ and $[ , , ]_M$ are simply denoted by $[ , , ]$, and $[, , ]_{\mu\rho\beta}$ is simply denoted by $[ , , ]_A.$

 \begin{lemma} Let $A$ be a $(\mu, \rho, \beta)$-extension of $H$ by $M$, and satisfy

 \begin{equation}
 \rho(x_4, [x_1, x_2, x_3])=\rho(x_3, x_1)\rho(x_4, x_2)-\rho(x_2, x_1)\rho(x_4, x_3)
 \label{eq:b2}
\end{equation}

\hspace{4cm}
$+\rho(x_2, x_3)\rho(x_4, x_1)-\beta(x_4, \mu(x_1, x_2, x_3)).$
\\Then Eq.(\ref{eq:rbn22}) holds if and only if
\begin{equation}
\rho(x_4, [x_1, x_2, x_3])=\rho(x_3, [x_1, x_2, x_4])-\beta(x_4, \mu(x_1, x_2, x_3))
 \label{eq:b3}
\end{equation}

\hspace{2cm}$+\beta(x_3, \mu(x_1, x_2, x_4))-\rho(x_1, x_2)\rho(x_3, x_4)+\rho(x_3, x_4)\rho(x_1, x_2).$

 \label{lem:le31}
 \end{lemma}

\begin{proof} If Eq.(\ref{eq:rbn22}) holds, by Eq.(\ref{eq:b2}),

\vspace{1mm}$\rho(x_4, [x_1, x_2, x_3])-\rho(x_3, [x_1, x_2, x_4])$
\\$=\rho(x_1, x_3)\rho(x_2, x_4)-\rho(x_1, x_2)\rho(x_3, x_4)$
$-\rho(x_2, x_3)\rho(x_1, x_4)$
\\
$-\beta(x_4, \mu(x_1, x_2, x_3))$$+\rho(x_2, x_4)\rho(x_1, x_3)-\rho(x_1, x_4)\rho(x_2, x_3)$
\\$-\rho(x_1, x_2)\rho(x_3, x_4)+\beta(x_3, \mu(x_1, x_2, x_4))$$-\rho(x_2, x_3)\rho(x_1, x_4)$
\\$=\rho(x_1, x_3)\rho(x_2, x_4)+ \rho(x_2, x_4)\rho(x_1, x_3)$
$-\rho(x_1, x_4)\rho(x_2, x_3)$
\\$-2\rho(x_1, x_2)\rho(x_3, x_4)-\beta(x_4, \mu(x_1, x_2, x_3))+$
$\beta(x_3, \mu(x_1, x_2, x_4))$
\\$=-\beta(x_4, \mu(x_1, x_2, x_3))+\beta(x_3, \mu(x_1, x_2, x_4))-\rho(x_1, x_2)\rho(x_3, x_4)$\\$+\rho(x_3, x_4)\rho(x_1, x_2).$

It follows  that Eq.(\ref{eq:b3}) holds.
Conversely, thanks to  Eq.(\ref{eq:b2})
\\
$\rho(x_3, [x_1, x_2, x_4])=-\rho(x_3, [x_4, x_2, x_1])$$=\rho(x_2, x_4)\rho(x_3, x_1)$
\\$-\rho(x_1, x_4)\rho(x_3, x_2)$
$+\rho(x_1, x_2)\rho(x_3, x_4)$
$-\beta(x_3, \mu(x_1, x_2, x_4)).$
\\
Then by Eq.(\ref{eq:b3}),
\\$\rho(x_4, [x_1, x_2, x_3])$\\
$=\rho(x_2, x_4)\rho(x_3, x_1)-\rho(x_1, x_4)\rho(x_3, x_2)$
$+\rho(x_1, x_2)\rho(x_3, x_4)$
\\$-\beta(x_3, \mu(x_1, x_2, x_4))-\beta(x_4, \mu(x_1, x_2, x_3))$
$+\beta(x_3, \mu(x_1, x_2, x_4))$\\
$-\rho(x_1, x_2)\rho(x_3, x_4)$
$+\rho(x_3, x_4)\rho(x_1, x_2)$
$=\rho(x_2, x_4)\rho(x_3, x_1)$\\
$-\rho(x_1, x_4)\rho(x_3, x_2)$
$-\beta(x_4, \mu(x_1, x_2, x_3))$
$+\rho(x_3, x_4)\rho(x_1, x_2)$
\\
$=\rho(x_1, x_4)\rho(x_2, x_3)+\rho(x_2, x_3)\rho(x_1, x_4)+\rho(x_1, x_2)\rho(x_3, x_4)$\\
$+\rho(x_3, x_4)\rho(x_1, x_2)$
$-\rho(x_1, x_3)\rho(x_2, x_4)-\rho(x_2, x_4)\rho(x_1, x_3)$\\
$+\rho(x_4, [x_1, x_2, x_3]),$
 we obtain
Eq.(\ref{eq:rbn22}).
\end{proof}

\begin{lemma} Let $A$ be a $(\mu, \rho, \beta)$-extension  of $H$ by $M$ satisfying, for all $x, y\in M$ and $h_1, h_2, h\in H$,
\begin{equation}
\beta(y, h_2)\beta(x, h_1)h-\beta(y, h)\beta(x, h_1)h_2-\beta(x, h_1)\beta(y, h_2)h,
\label{eq:b4}
\end{equation}
\hspace{2cm}$=[\rho(x, y)h_1, h_2, h].$
\\Then we have
\begin{equation}
\rho(x,y)[h_1, h_2, h]+\beta(y, h_1)\beta(x, h_2)h-\beta(x, h_1)\beta(y, h_2)h
\label{eq:b5}
\end{equation}
\hspace{2cm}$=[\rho(x, y)h_1, h_2, h]$.
\label{lem:le32}

\end{lemma}

\begin{proof}  Thanks to Eq.(\ref{eq:b4}) and $\rho(x, y)\in Der(H)$,

\vspace{1mm}$[h_1, \rho(x, y)h_2, h]$\\
$=\beta(y, h_1)\beta(x, h)h_2+\beta(y, h)\beta(x, h_2)h_1+\beta(x, h_2)\beta(y, h_1)h,$

$[h_1, h_2, \rho(x, y)h]$\\
$=\beta(y, h_2)\beta(x, h_1)h+\beta(y, h_1)\beta(x, h)h_2+\beta(x, h)\beta(y, h_2)h_1,$

\vspace{1mm} $\rho(x, y)[h_1, h_2, h]$\\
$=2(\beta(y, h_1)\beta(x, h)h_2+\beta(y, h_2)\beta(x, h_1)h+\beta(y, h)\beta(x, h_2)h_1)$
\\$+\beta(x, h_1)\beta(y, h)h_2+\beta(x, h_2)\beta(y, h_1)h+\beta(x, h)\beta(y, h_2)h_1.$

Therefore

\vspace{1mm}$\beta(y, h_1)\beta(x, h)h_2+\beta(y, h_2)\beta(x, h_1)h+\beta(y, h)\beta(x, h_2)h_1$
\\
$+\beta(x, h_1)\beta(y, h)h_2$
$+\beta(x, h_2)\beta(y, h_1)h+\beta(x, h)\beta(y, h_2)h_1=0,$

\vspace{1mm}$\rho(x, y)[h_1, h_2, h]$\\
$=\beta(y, h_1)\beta(x, h)h_2+\beta(y, h_2)\beta(x, h_1)h+\beta(y, h)\beta(x_1, h_2)h_1,$
\\
that is,

$\beta(y, h_2)\beta(x, h_1)h-\beta(y, h)\beta(x, h_1)h_2$\\
$=\rho(x, y)[h_1, h_2, h]+\beta(y, h_1)\beta(x, h_2)h.$

 Again by  Eq.(\ref{eq:b4}), we obtain  Eq.(\ref{eq:b5}).

\end{proof}

\begin{lemma} Let $A$ be a $(\mu, \rho, \beta)$-extension of $H$ by $M$ satisfying for all $x\in M$ and $h_1, h_2, h_3\in H$,
\begin{equation}
ad(\beta(x, h_1)h_3, h_2)+ad(h_3, \beta(x, h_1)h_2)+ad(\beta(x, h_3)h_2, h_1),
\label{eq:b6}
\end{equation}
\hspace{2cm}$=\beta(x, [h_1, h_2, h_3]).$
\\Then we
\begin{equation}[h_1, h_2, \beta(x, h_3)h_4]-\beta(x, [h_1, h_2, h_3])h_4-[h_3, h_4, \beta(x, h_1)h_2].
\label{eq:b7}
\end{equation}
\hspace{2cm}$=\beta(x, h_3)[h_1, h_2, h_4].$
\label{lem:le33}
\end{lemma}

\begin{proof} Since $\beta(x, h_3)\in Der(H),$  $x\in M$ and $h_1, h_2, h_3, h_4\in H$, we have

$\beta(x, [h_1, h_2, h_3])h_4$\\
$=[\beta(x, h_1)h_3, h_2, h_4]+[h_3, \beta(x, h_1)h_2, h_4]+[\beta(x, h_3)h_2, h_1, h_4].$
\\
$=-[\beta(x, h_3)h_1, h_2, h_4]+[h_3, \beta(x, h_1)h_2, h_4]+[\beta(x, h_3)h_2, h_1, h_4]$
\\$=-\beta(x, h_3)[h_1, h_2, h_4]+[h_1, h_2, \beta(x, h_3)h_4]+[h_3, \beta(x, h_1)h_2, h_4].$
\\It follows that Eq.(\ref{eq:b7}) holds.
\end{proof}

\begin{theorem} Let $A=M\dot+H$ be a $(\mu, \rho, \beta)$-extension  of $H$ by $M$. Then $A$ is a $3$-Lie algebra if and only if for all $x_1, x_2, x_3, x_4, x_5\in M$ and
$h_1, h_2\in H$,  identities
Eq.(\ref{eq:rbn22}), Eq.(\ref{eq:b2}), Eq.(\ref{eq:b4}), Eq.(\ref{eq:b6}) and the following identities hold
\begin{equation}
[\mu(x_1, x_2, x_3), h_1, h_2]=\rho(x_2, x_3)\beta(x_1, h_1)h_2-\rho(x_1, x_3)\beta(x_2, h_1)h_2
\label{eq:b8}
\end{equation}
\hspace{5cm}$+\rho(x_1, x_2)\beta(x_3, h_1)h_2-\beta([x_1, x_2, x_3], h_1)h_2,$
\begin{equation}
\beta(x_1, h_1)\rho(x_2, x_3)h_2+\beta(x_3, h_2)\rho(x_1, x_2)h_1
\label{eq:b9}
\end{equation}
\hspace{3cm}$=\rho(x_2, x_3)\beta(x_1, h_1)h_2
+\beta(x_2, h_2)\rho(x_1, x_3)h_1,$
\begin{equation}
\mu(x_1, x_2, [x_3, x_4, x_5])-\mu([x_1, x_2, x_3], x_4, x_5)-\mu(x_3, [x_1, x_2, x_4], x_5)
\label{eq:b10}
\end{equation}

\hspace{5mm}$-\mu(x_3, x_4, [x_1, x_2, x_5])=\rho(x_3, x_4)\mu(x_1, x_2, x_5)-\rho(x_3, x_5)\mu(x_1, x_2, x_4)$

\hspace{5mm}$-\rho(x_1, x_2)\mu(x_3, x_4, x_5)+\rho(x_4, x_5)\mu(x_1, x_2, x_3).$
\label{thm:t31}
\end{theorem}

\begin{proof} If $A$ is a $3$-Lie algebra,  then Eqs.(\ref{eq:b4}), (\ref{eq:b8}), (\ref{eq:b9}) and (\ref{eq:b10}) follow from Eq.(\ref{eq:b1}), directly.
Now, we prove that Eqs.(\ref{eq:rbn22}), (\ref{eq:b2}), (\ref{eq:b6}) hold.
For all $x_i\in M, ~  i=1, \cdots, 4, ~ h\in H$, by Eq.(\ref{eq:b1}),

$[h, x_1, [x_2, x_3, x_4]_A]_A=[h, x_1, [x_2, x_3, x_4]]+[h, x_1, \mu(x_2, x_3, x_4)]$
\\$=\rho(x_1, [x_2, x_3, x_4])h+\beta(x_1, \mu(x_2, x_3, x_4))h,$

$[[h, x_1, x_2]_A, x_3, x_4]_A+[x_2, [h, x_1, x_3]_A]_A+[x_2, x_3, [h, x_1, x_4]_A]_A$
\\$=\rho(x_3, x_4)\rho(x_1, x_2)h-\rho(x_2, x_4)\rho(x_1, x_3)h+\rho(x_2, x_3)\rho(x_1, x_4)h.$

Then we have

$\rho(x_1, [x_2, x_3, x_4])+\beta(x_1, \mu(x_2, x_3, x_4))$
\\
$=\rho(x_3, x_4)\rho(x_1, x_2)-\rho(x_2, x_4)\rho(x_1, x_3)+\rho(x_2, x_3)\rho(x_1, x_4).$

Replacing $x_1, x_2, x_3, x_4$ by $x_4, x_1, x_2, x_3$, we obtain Eq.(\ref{eq:b2}).

From $[x_1, x_2, [x_3, x_4, h]_{A}]_{A}=\rho(x_1, x_2)\rho(x_3, x_4)h,$ and

$[[x_1, x_2, x_3]_{A}, x_4, h]_{A}+[x_3, [x_1, x_2, x_4]_{A}, h]_{A}+[x_3, x_4, [x_1, x_2, h]_{A}]_{A}$
\\
$=\rho([x_1, x_2, x_3], x_4)h-\beta(x_4, \mu(x_1, x_2, x_3))h+\rho(x_3, [x_1, x_2, x_4])h$
\\
$+\beta(x_3, \mu(x_1, x_2, x_4))h+\rho(x_3, x_4)\rho(x_1, x_2)h,$
\\ we obtain Eq.(\ref{eq:b3}). By Lemma \ref{lem:le31}, Eq.(\ref{eq:rbn22}) holds.

For all $ h_i\in H, i=1, 2, 3, 4$ and $x\in M,$ by Eq.(\ref{eq:b1}),

$[h_1, h_2, [h_3, h_4, x]_A]_A=[h_1, h_2, \beta(x, h_3)h_4],$ and

\vspace{1mm}$[[h_1, h_2, h_3]_A, h_4, x]_A+[h_3, [h_1, h_2, h_4]_A, x]_A+[h_3, h_4, [h_1, h_2, x]_A]_A$
\\
$=\beta(x, [h_1, h_2, h_3])h_4+\beta(x, h_3)[h_1, h_2, h_4]+[h_3, h_4, \beta(x, h_1)h_2].$

Then we have

$[h_1, h_2, \beta(x, h_3)h_4]$
\\
$=\beta(x, [h_1, h_2, h_3])h_4$ $+\beta(x, h_3)[h_1, h_2, h_4]$ $+[h_3, h_4, \beta(x, h_1)h_2].$

Since $\beta(x, h_3)\in Der(H)$, we obtain

\vspace{1mm}$\beta(x, [h_1, h_2, h_3])h_4$
\\
$=[h_1, h_2, \beta(x, h_3)h_4]-\beta(x, h_3)[h_1, h_2, h_4]-[h_3, h_4, \beta(x, h_1)h_2]$
\\
$=-[\beta(x, h_3)h_1, h_2, h_4]-[h_1, \beta(x, h_3)h_2, h_4]+[h_3, \beta(x, h_1)h_2, h_4]$
\\
$=ad(\beta(x, h_1)h_3, h_2)h_4+ad(h_3, \beta(x, h_1)h_2)h_4+ad(\beta(x, h_3)h_2, h_1)h_4.$
\\
It follows that Eq.(\ref{eq:b6}) holds.

Conversely, to prove $A$ is a $3$-Lie algebra,
 we only need to show  that the multiplication Eq.(\ref{eq:b1}) satisfies Jacobi identities. So we divide our discussion into cases.

{\bf Case 1.}  For all $x_i\in M, i=1, \cdots, 5$,

 $[x_1, x_2, [x_3, x_4, x_5]_{A}]_{A}$
\\
$=[x_1, x_2, [x_3, x_4, x_5]]+\mu(x_1, x_2, [x_3, x_4, x_5])+\rho(x_1, x_2)\mu(x_3, x_4, x_5),$

$[[x_1, x_2, x_3]_{A}, x_4, x_5]_{A}+[x_3, [x_1, x_2, x_4]_{A}, x_5]_{A}+[x_3, x_4, [x_1, x_2, x_5]_{A}]_{A}$
\\
$=[[x_1, x_2, x_3], x_4, x_5]+\mu([x_1, x_2, x_3], x_4, x_5)+\rho(x_4, x_5)\mu(x_1, x_2, x_3)$
\\
$+[x_3, [x_1, x_2, x_4], x_5]+\mu(x_3, [x_1, x_2, x_4], x_5)+\rho(x_5, x_3)\mu(x_1, x_2, x_4)$
\\
$+[x_3, x_4, [x_1, x_2, x_5]]+\mu(x_3, x_4, [x_1, x_2, x_5])+\rho(x_3, x_4)\mu(x_1, x_2, x_5).$

Thanks to Eq.(\ref{eq:b10}), we have

$[x_1, x_2, [x_3, x_4, x_5]_{A}]_{A}$
\\
$=[[x_1, x_2, x_3]_{A}, x_4, x_5]_{A}+[x_3, [x_1, x_2, x_4]_{A}, x_5]_{A}+[x_3, x_4, [x_1, x_2, x_5]_{A}]_{A}.$

\vspace{2mm} {\bf Case 2.}  For all  $x_i\in M, i=1, \cdots, 4, h\in H$,

$[h, x_1, [x_2, x_3, x_4]_A]_A=[h, x_1, [x_2, x_3, x_4]]+[h, x_1, \mu(x_2, x_3, x_4)]$
\\$=\rho(x_1, [x_2, x_3, x_4])h+\beta(x_1, \mu(x_2, x_3, x_4))h,$

$[[h, x_1, x_2]_A, x_3, x_4]_A+[x_2, [h, x_1, x_3]_A]_A+[x_2, x_3, [h, x_1, x_4]_A]_A$
\\
$=\rho(x_3, x_4)\rho(x_1, x_2)h-\rho(x_2, x_4)\rho(x_1, x_3)h+\rho(x_2, x_3)\rho(x_1, x_4)h.$
\\
Replacing $x_4, x_1, x_2, x_3$  by $x_1, x_2, x_3, x_4$ in Eq.(\ref{eq:b2}), we obtain

\vspace{1mm}$\rho(x_1, [x_2, x_3, x_4])=\rho(x_3, x_4)\rho(x_1, x_2) $
$-\rho(x_2, x_4)\rho(x_1, x_3)$
\\$+\rho(x_2, x_3)\rho(x_1, x_4)-\beta(x_1, \mu(x_2, x_3, x_4)).$

Therefore,

\vspace{1mm}$[h, x_1, [x_2, x_3, x_4]_A]_A$
\\
$=[[h, x_1, x_2]_A, x_3, x_4]_A+[x_2, [h, x_1, x_3]_A]_A+[x_2, x_3, [h, x_1, x_4]_A]_A.$

In addition, $[x_1, x_2, [x_3, x_4, h]_{A}]_{A}=\rho(x_1, x_2)\rho(x_3, x_4)h,$ and

\vspace{1mm}$[[x_1, x_2, x_3]_{A}, x_4, h]_{A}+[x_3, [x_1, x_2, x_4]_{A}, h]_{A}+[x_3, x_4, [x_1, x_2, h]_{A}]_{A}$
\\$=\rho([x_1, x_2, x_3], x_4)h-\beta(x_4, \mu(x_1, x_2, x_3))h+\rho(x_3, [x_1, x_2, x_4])h$
\\
$+\beta(x_3, \mu(x_1, x_2, x_4))h+\rho(x_3, x_4)\rho(x_1, x_2)h.$

By Lemma \ref{lem:le31} and  Eq.(\ref{eq:b2}),

$\rho(x_4, [x_1, x_2, x_3])-\rho(x_3, [x_1, x_2, x_4])$
\\
$=\rho(x_2, x_3)\rho(x_4, x_1)-\rho(x_1, x_3)\rho(x_4, x_2)+\rho(x_1, x_2)\rho(x_4, x_3)$
\\
$-\beta(x_4, \mu(x_1, x_2, x_3))-\rho(x_2, x_4)\rho(x_3, x_1)+\rho(x_1, x_4)\rho(x_3, x_2) x_4))$
\\
$-\rho(x_1, x_2)\rho(x_3, x_4)+\beta(x_3, \mu(x_1, x_2, x_4)$
\\
$=\beta(x_3, \mu(x_1, x_2, x_4))-\beta(x_4, \mu(x_1, x_2, x_3))-\rho(x_2, x_3)\rho(x_1, x_4)$
\\
$-\rho(x_1, x_4)\rho(x_2, x_3)+\rho(x_1, x_3)\rho(x_2, x_4)+\rho(x_2, x_4)\rho(x_1, x_3)$
\\
$-2\rho(x_1, x_2)\rho(x_3, x_4).$
By Eq.(\ref{eq:rbn22}) and Eq.(\ref{eq:b3}),

$[x_1, x_2, [x_3, x_4, h]_{A}]_{A}$
\\
$=[[x_1, x_2, x_3]_{A}, x_4, h]_{A}+[x_3, [x_1, x_2, x_4]_{A}, h]_{A}+[x_3, x_4, [x_1, x_2, h]_{A}]_{A}.$

\vspace{2mm}{\bf Case 3.}  For all $x_i\in M,  i=1, 2, 3,$  and $h_1, h_2\in H,$ by By Eq.(\ref{eq:b9})

$[x_1, h_1, [x_2, x_3, h_2]_A]_A$
\\
$=[[x_1, h_1, x_2]_A, x_3, h_2]_A+[x_2, [x_1, h_1, x_3]_A, h_2]_A+[x_2, x_3, [x_1, h_1, h_2]_A]_A.$

 By Eq.(\ref{eq:b8})

$[h_1, h_2, [x_1, x_2, x_3]_A]_A$
\\
$=[[h_1, h_2, x_1]_A, x_2, x_3]_A+[x_1, [h_1, h_2, x_2]_A, x_3]_A+[x_1, x_2, [h_1, h_2, x_3]_A]_A$,

 $[x_1, x_2, [x_3, h_1, h_2]_A]_A=\rho(x_1, x_2)\beta(x_3, h_1)h_2,$

$[[x_1, x_2, x_3]_A, h_1, h_2]_A+[x_3, [x_1, x_2, h_1]_A, h_2]_A+[x_3, h_1, [x_1, x_2, h_2]_A]_A$
\\$=\beta([x_1, x_2, x_3], h_1)h_2+\beta(x_3, \rho(x_1, x_2)h_1)h_2+\beta(x_3, h_1)\rho(x_1, x_2)h_2$
\\
$+[\mu(x_1, x_2, x_3), h_1, h_2].$

By a direct computation and By Eq.(\ref{eq:b8}) and By Eq.(\ref{eq:b9}),

\vspace{1mm}$[\mu(x_1, x_2, x_3), h_1, h_2]$
\\
$=\beta(x_1, h_1)\rho(x_2, x_3)h_2+\beta(x_3, h_2)\rho(x_1, x_2)h_1$
$-\beta(x_2, h_2)\rho(x_1, x_3)h_1$
\\
$+\rho(x_1, x_3)\beta(x_2, h_2)h_1$
$-\rho(x_1, x_2)\beta(x_3, h_2)h_1-\beta([x_1, x_2, x_3], h_1)h_2,$

\vspace{1mm}$[\mu(x_1, x_2, x_3), h_1, h_2]$
\\
$=\beta(x_2, h_2)\rho(x_1, x_3)h_1+\beta(x_3, h_1)\rho(x_2, x_1)h_2$
$-\beta(x_1, h_1)\rho(x_2, x_3)h_2$
\\
$+\rho(x_2, x_3)\beta(x_1, h_1)h_2$
$-\rho(x_2, x_1)\beta(x_3, h_1)h_2-\beta([x_1, x_2, x_3], h_1)h_2,$

\vspace{1mm}$[\mu(x_1, x_2, x_3), h_1, h_2]$
\\
$=\rho(x_1, x_2)\beta(x_3, h_1)h_2-\beta(x_3, h_1)\rho(x_1, x_2)h_2$
$+\beta(x_3, h_2) \rho(x_1, x_2)h_1$
\\
$-\beta([x_1, x_2, x_3], h_1)h_2.$

Therefore,

$[x_1, x_2, [x_3, h_1, h_2]_A]_A$
\\
$=[[x_1, x_2, x_3]_A, h_1, h_2]_A+[x_3, [x_1, x_2, h_1]_A, h_2]_A+[x_3, h_1, [x_1, x_2, h_2]_A]_A.$

\vspace{2mm}{\bf Case 4.}  For all $x_1, x_2\in M, h_i\in H, i=1, 2, 3,$ since $\rho(x_1, x_2)\in Der(H)$ and Eq.(\ref{eq:b4}), we obtain

\vspace{1mm}$[x_1, x_2, [h_1, h_2, h_3]_A]_A$
\\
$=[[x_1, x_2, h_1]_A, h_2, h_3]_A+[h_1, [x_1, x_2, h_2]_A, h_3]_A+[h_1, h_2, [x_1, x_2, h_3]_A]_A,$

$[x_1, h_1, [x_2, h_2, h_3]_A]_A$
\\
$=[[x_1, h_1, x_2]_A, h_2, h_3]_A+[x_2, [x_1, h_1, h_2]_A, h_3]_A+[x_2, h_2, [x_1, h_1, h_3]_A]_A.$

We know that $[h_1, h_2, [h_3, x_1, x_2]_A]_A=[h_1, h_2, \rho(x_1, x_2)h_3],$
and

 $[[h_1, h_2, h_3]_A, x_1, x_2]_A+[h_3, [h_1, h_2, x_1]_A, x_2]_A+[h_3, x_1, [h_1, h_2, x_2]_A]_A$
\\
$=\rho(x_1, x_2)[h_1, h_2, h_3]+\beta(x_2, h_3)\beta(x_1, h_1)h_2-\beta(x_1, h_3)\beta(x_2, h_1)h_2$

Thanks to  Lemma \ref{lem:le32},

\vspace{1mm}$\rho(x_1, x_2)[h_1, h_2, h_3]$
\\
$=\beta(x_2, h_1)\beta(x_1, h_3)h_2+\beta(x_2, h_2)\beta(x_1, h_1)h_3+\beta(x_2, h_3)\beta(x_1, h_2)h_1,$

\vspace{1mm}$\beta(x_2, h_2)\beta(x_1, h_1)h_3-\beta(x_2, h_3)\beta(x_1, h_1)h_2$
\\
$=\rho(x_1, x_2)[h_1, h_2, h_3]+\beta(x_2, h_1)\beta(x_1, h_2)h_3.$

Again by Eq.(\ref{eq:b4})

\vspace{1mm}$[\rho(x_1, x_2)h_1, h_2, h_3]$
\\
$=\rho(x_1, x_2)[h_1, h_2, h_3]+\beta(x_2, h_1)\beta(x_1, h_2)h_3-\beta(x_1, h_1)\beta(x_2, h_2)h_3.$

It follows that
$[h_1, h_2, [h_3, x_1, x_2]_A]_A$
\\
$=[[h_1, h_2, h_3]_A, x_1, x_2]_A+[h_3, [h_1, h_2, x_1]_A, x_2]_A+[h_3, x_1, [h_1, h_2, x_2]_A]_A.$

\vspace{2mm}{\bf Case 5.}  For all $x\in M, h_i\in H, i=1, 2, 3, 4$, since $\beta(x, h_1)\in Der(H)$ and Eq.(\ref{eq:b6}),

\vspace{1mm}$[x, h_1, [h_2, h_3, h_4]_A]_A$
\\
$=[[x, h_1, h_2]_A, h_3, h_4]_A+[h_2, [x, h_1, h_3]_A, h_4]_A+[h_2, h_3, [x, h_1, h_4]_A]_A.$

$[h_1, h_2, [h_3, h_4, x]_A]_A$
\\
$=[[h_1, h_2, h_3]_A, h_4, x]_A+[h_3, [h_1, h_2, h_4]_A, x]_A+[h_3, h_4, [h_1, h_2, x]_A]_A.$

Summarizing the above discussion, we obtain that the multiplication defined by Eq.(\ref{eq:b1}) satisfies Jacobi identities.
The proof is complete.
\end{proof}

\begin{theorem} Let  $A=H\dot+ M$ be a
$(\mu, \rho, \beta)$-extension $3$-Lie algebra of $H$ by $M$. Then linear maps $i: H\rightarrow A$  and $p: A\rightarrow M$ defined by  for all $h\in H$, $x\in M$,
$i(h)=h$ and $p(h+x)=x$,  are Lie homomorphisms, and  the following sequence
is exact

  \begin{equation}
    \begin{CD}
        0   @>>>     H    @>i>>       A       @>p>>   M       @>>>    0.
          \end{CD}
\label{eq:b11}
\end{equation}
\label{cor:t32}
\end{theorem}

\begin{proof} The result follows from  a direct computation according to Theorem \ref{thm:t31} and Eq.(\ref{eq:b1}).
\end{proof}

\begin{theorem} Let $A=M\dot+H$ be a $(\mu, \rho, \beta)$-extension $3$-Lie algebra of $H$ by $M$. Then $(H, \rho)$ is an $M$-module if and only if
\vspace{1mm}$\beta(M, \mu(M, M, M))=0.$
\label{thm:t32}

\end{theorem}

\begin{proof}

If $\beta(M, \mu(M, M, M))=0$, by Lemmas \ref{lem:le21} and \ref{lem:le31}, $(H, \rho)$ is an $M$-module.

Conversely, for all $x_j\in M, j=1, 2, 3, 4$, from Theorem \ref{thm:t31}, and Eqs.(\ref{eq:b2}) and (\ref{eq:b3})

$\rho([x_1, x_2, x_4], x_3)-\rho([x_1, x_2, x_3], x_4)$
\\
$=-\beta(x_4, \mu(x_1, x_2, x_3))+\beta(x_3, \mu(x_1, x_2, x_4))$
$-\rho(x_1, x_2)\rho(x_3, x_4)$
\\
$+\rho(x_3, x_4)\rho(x_1, x_2),$ and

$\rho(x_2, [x_1, x_3, x_4])$
\\
$=\rho(x_3, x_4)\rho(x_2, x_1)-\rho(x_1, x_4)\rho(x_2, x_3)+\rho(x_1, x_3)\rho(x_2, x_4)$
\\
$-\beta(x_2, \mu(x_1, x_3, x_4)).$
Therefore,

$-\rho([x_1, x_2, x_3], x_4)+\rho([x_1, x_2, x_4], x_3)+\rho(x_2, [x_1, x_3, x_4])-\rho(x_1, [x_2, x_3, x_4])$
\\$=-\beta(x_4, \mu(x_1, x_2, x_3))+\beta(x_3, \mu(x_1, x_2, x_4))-\rho(x_1, x_2)\rho(x_3, x_4)$
\\$+\rho(x_3, x_4)\rho(x_1, x_2)+\rho(x_3, x_4)\rho(x_2, x_1)-\rho(x_1, x_4)\rho(x_2, x_3)$
\\$+\rho(x_1, x_3)\rho(x_2, x_4)-\beta(x_2, \mu(x_1, x_3, x_4))$$-\rho(x_3, x_4)\rho(x_1, x_2)$
\\$+\rho(x_2, x_4)\rho(x_1, x_3)-\rho(x_2, x_3)\rho(x_1, x_4)+\beta(x_1, \mu(x_2, x_3, x_4))$
\\$=\beta(x_3, \mu(x_1, x_2, x_4))-\beta(x_4, \mu(x_1, x_2, x_3))+\beta(x_1, \mu(x_2, x_3, x_4))$
\\$-\beta(x_2, \mu(x_1, x_3, x_4))$
$-\rho(x_1, x_2)\rho(x_3, x_4)-\rho(x_3, x_4)\rho(x_1, x_2)$
\\$+\rho(x_1, x_3)\rho(x_2, x_4)$$+\rho(x_2, x_4)\rho(x_1, x_3)$
$-\rho(x_1, x_4)\rho(x_2, x_3)$
\\$-\rho(x_2, x_3)\rho(x_1, x_4)$
\\$=\beta(x_3, \mu(x_1, x_2, x_4))-\beta(x_4, \mu(x_1, x_2, x_3))+\beta(x_1, \mu(x_2, x_3, x_4))$
\\$-\beta(x_2, \mu(x_1, x_3, x_4))$
$=0.$
\\
By Eq.(\ref{eq:b2}), $\beta(x_4, \mu(x_1, x_2, x_3))=0.$ The proof is complete.
\end{proof}

\begin{theorem} Let $A=M\dot+H$ be a $(\mu, \rho)$-extension of $H$ by $M$, and $(H, \rho)$ be an $M$-module. Then $A$ is a $3$-Lie algebra  if and only if   $\mu(M, M, M)\subseteq  Z(H)$, $\rho(M, M)\subseteq Z(Der(H))$ and Eq.(\ref{eq:b10}) holds.

\label{thm:t33}
\end{theorem}

\begin{proof} If $A$ is a $3$-Lie algebra. By   Theorem \ref{thm:t31}, Eq.(\ref{eq:b10}) holds, and from  $\beta=0$ and Eq.(\ref{eq:b6}) and  Eq.(\ref{eq:b2}), we get  $\mu(M, M, M)\subseteq  Z(H)$, $\rho(M, M)\subseteq Z(Der(H))$, respectively. Conversely, the result follows from Theorem \ref{thm:t32}  and Theorem \ref{thm:t31}, directly .

\end{proof}

\section{Extensions of derivations}

In this section we suppose that $M$ and $H$ are $3$-Lie algebras, $A$ is a $(\mu, \rho, \beta)$-extension $3$-Lie algebra of  $H$ by  $M$.
For  derivations  $\sigma\in Der M,$ $ \tau\in Der H$ and $\delta\in Der A$, if the following diagram commutes,

\begin{equation}
\begin{CD}
        0   @>>>     H    @>i>>       A       @>p>>   M       @>>>    0   \\
        @.              @V\tau VV           @V\delta VV      @V\sigma VV        \\
        0   @>>>     H    @>i>>       A       @>p>>   M     @>>>    0.
    \end{CD}
 \label{eq:b41}
 \end{equation}
then  we say that the
pair $(\sigma, \tau)\in Der(M)\times Der(H)$ extends to a derivation $\delta \in Der(A)$,
where $i(h)=h$ and $p(x+h)=x$ for all $x\in M$ and $h\in H$.
And   $(\sigma, \tau)$ is called {\bf extendable.}

 We first introduce  the external  direct sum  $3$-Lie algebra of a given $3$-Lie algebra.

 Let $A$ be a vector space, the vector space $\{ (x_1, x_2, x_3)~~ |~~ \forall  x_1, x_2, x_3\in A ~ \}$ is called the exterior direct sum space of $A$, and is denoted by $A^3.$

\begin{theorem}  Let $A$ be an arbitrary $3$-Lie algebra. Then  $A^3$ is a
 $3$-Lie algebra in the following multiplication, for all $x_i, y_i, z_i\in A$, $i=1, 2, 3,$

 \begin{equation}
[(x_1, y_1, z_1), (x_2, y_2, z_2), (x_3, y_3, z_3)]
  \label{eq:b42}
\end{equation}
$$=([x_1, y_2, y_3]+[x_2, y_3, y_1]+[x_3, y_1, y_2],[y_1, y_2, y_3],[z_1, z_2, z_3]).
$$
\end{theorem}

And it  is called {\bf the  exterior direct sum  $3$-Lie algebra of $A$}, and is simply denoted by $A^3$.
\label{thm:t41}

\begin{proof} It is clear that the  multiplication  Eq.(\ref{eq:b42}) is skew-symmetric.
For all $ x_j, y_j, z_j$ in $A$ for $1\leq j\leq 5,$ we have

\vspace{1mm}$[(x_1, y_1, z_1), (x_2, y_2, z_2), [(x_3, y_3, z_3), (x_4, y_4, z_4), (x_5, y_5, z_5)]]$
\\
$=[(x_1, y_1, z_1), (x_2, y_2, z_2), ([x_3, y_4, y_5]+[x_5, y_3, y_4]$
\\
$+[x_4, y_5, y_3], [y_3, y_4, y_5], [z_3, z_4, z_5])]$
\\
$=([x_1, y_2, [y_3, y_4, y_5]]+[[x_3, y_4, y_5]+[x_5, y_3, y_4]+[x_4, y_5, y_3], y_1, y_2]$
\\
$+[x_2, [y_3, y_4, y_5], y_1], [y_1, y_2, [y_3, y_4, y_5]], [z_1, z_2, [z_3, z_4, z_5]])$
\\
$=([[x_1, y_2, y_3], y_4, y_5]+[y_3, [x_1, y_2, y_4], y_5]+[y_3, y_4, [x_1, y_2, y_5]]$
\\
$+[[x_3, y_1, y_2], y_4, y_5]+[x_3, [y_4, y_1, y_2], y_5]+[x_3, y_4, [y_5, y_1, y_2]]$
\\
$+[[x_5, y_1, y_2], y_3, y_4]+[x_5, [y_3, y_1, y_2], y_4]+[x_5, y_3, [y_4, y_1, y_2]]$
\\
$+[[x_4, y_1, y_2], y_5, y_3]+[x_4, [y_5, y_1, y_2], y_3]+[x_4, y_5, [y_3, y_1, y_2]]$
\\
$+[[x_2, y_3, y_1], y_4, y_5]$
$+[y_3, [x_2, y_4, y_1], y_5]+[y_3, y_4, [x_2, y_5, y_1]],$
\\
$[y_1, y_2, [y_3, y_4, y_5]], ~ [z_1, z_2, [z_3, z_4, z_5]]),$

\vspace{2mm}$[[(x_1, y_1, z_1), ~ (x_2, y_2, z_2), ~ (x_3, y_3, z_3)], ~ (x_4, y_4, z_4),~  (x_5, y_5, z_5)]$
\\
$+[(x_3, y_3, z_3), ~ [(x_1, y_1, z_1), (x_2, y_2, z_2), (x_4, y_4, z_4)], ~ (x_5, y_5, z_5)]$
\\
$+[(x_3, y_3, z_3), (x_4, y_4, z_4), [(x_1, y_1, z_1), (x_2, y_2, z_2), (x_5, y_5, z_5)]]$
\\
$=[([x_1, y_2, y_3]+[x_3, y_1, y_2]+[x_2, y_3, y_1], [y_1, y_2, y_3], [z_1, z_2, z_3]),$
\\
$ (x_4, y_4, z_4), (x_5, y_5, z_5)]$$+[(x_3, y_3, z_3), ([x_1, y_2, y_4]+[x_4, y_1, y_2]$
\\
$+[x_2, y_4, y_1], [y_1, y_2, y_4], [z_1, z_2, z_4]), (x_5, y_5, z_5)]$
\\
$+[(x_3, y_3, z_3)], $$(x_4, y_4, z_4), ([x_1, y_2, y_5]+[x_5, y_1, y_2]$
\\
$+[x_2, y_5, y_1], [y_1, y_2, y_5], [z_1, z_2, z_5])]$
\\
$=([[x_1, y_2, y_3]+[x_3, y_1, y_2]+[x_2, y_3, y_1], y_4, y_5]+[x_5, [y_1, y_2, y_3], y_4]$
\\
$+[x_4, y_5, [y_1, y_2, y_3]],$$[[y_1, y_2, y_3], y_4, y_5], [[z_1, z_2, z_3], z_4, z_5])$
\\
$+([x_3, [y_1, y_2, y_4], y_5]+[x_5, y_3, [y_1, y_2, y_4]]+[[x_1, y_2, y_4]$
\\
$+[x_4, y_1, y_2]+[x_2, y_4, y_1], y_5, y_3],$
$[y_3, [y_1, y_2, y_4], y_5], [z_3, [z_1, z_2, z_4], z_5])$
\\
$+([x_3, y_4, [y_1, y_2, y_5]]+[[x_1, y_2, y_5]+[x_5, y_1, y_2]+[x_2, y_5, y_1], y_3, y_4]$
\\
$+[x_4, [y_1, y_2, y_5], y_3],$$[y_3, y_4, [y_1, y_2, y_5]], [z_3, z_4, [z_1, z_2, z_5]])$
\\
$=([[x_1, y_2, y_3], y_4, y_5]+[[x_3, y_1, y_2], y_4, y_5]+[[x_2, y_3, y_1], y_4, y_5]$
\\
$+[x_5, [y_1, y_2, y_3], y_4]$
$+[x_4, y_5, [y_1, y_2, y_3]]+[x_3, [y_1, y_2, y_4], y_5]$
\\
$+[x_5, y_3, [y_1, y_2, y_4]]+[[x_1, y_2, y_4], y_5, y_3]$
$+[[x_4, y_1, y_2], y_5, y_3]$
\\
$+[[x_2, y_4, y_1], y_5, y_3]+[x_3, y_4, [y_1, y_2, y_5]]+[[x_1, y_2, y_5], y_3, y_4]$
\\
$+[[x_5, y_1, y_2], y_3, y_4]+[[x_2, y_5, y_1], y_3, y_4]$
\\
$+[x_4, [y_1, y_2, y_5], y_3], [[y_1, y_2, y_3], y_4, y_5]$
$+[y_3, [y_1, y_2, y_4], y_5]$
\\
$+[y_3, y_4, [y_1, y_2, y_5]], [[z_1, z_2, z_3], z_4, z_5]+[z_3, [z_1, z_2, z_4], z_5]$
\\
$+[z_3, z_4, [z_1, z_2, z_5]]).$

 Therefore $A^3$ is a $3$-Lie algebra in the multiplication Eq.(\ref{eq:b42}).
\end{proof}

\begin{theorem}  The subspace  $(0, 0, A)$ is an abelian ideal of the exterior direct sum $3$-Lie algebra $A^3$, and subspaces $(0, A, 0)$,  $ (A, 0, 0)$, $ (A, A, 0)$ and  $(0,  A, A)$ are subalgebras.
\label{thm:t42}
\end{theorem}

\begin{proof} The result follows from a direct computation.
\end{proof}

\begin{theorem} Let $A=M\dot+H$ be a $(\mu, \rho, \beta)$-extension $3$-Lie algebra of $H$ by $M$ and $(\sigma, \tau)\in Der(M)\times Der(H)$. Then the pair $(\sigma, \tau)$ is extendable if and only if there exists a linear map $\gamma: M\rightarrow H$ satisfying, for all $x, x_j\in M$ and $h\in H,$ $j=1, 2, 3,$

\begin{equation}
\tau\mu(x_1, x_2, x_3)+\gamma[x_1, x_2, x_3]
\label{eq:b43}
\end{equation}

\hspace{3cm}$=\mu(\sigma x_1, x_2, x_3)+\mu(x_1, \sigma x_2, x_3)+\mu(x_1, x_2, \sigma x_3)$

\hspace{3cm}$+\rho(x_1, x_2)\gamma(x_3)+\rho(x_3, x_1)\gamma(x_2)+\rho(x_2, x_3)\gamma(x_1),$

\begin{equation}
[\tau, \rho(x_1, x_2)]=\rho(\sigma x_1, x_2)+\rho(x_1, \sigma x_2)+\beta(\gamma(x_1), x_2)+\beta(x_1, \gamma(x_2)),
\label{eq:b44}
\end{equation}

\begin{equation}
[\tau, \beta(x, h)]=\beta(\sigma x, h)+\beta(x, \tau h)+ad(\gamma(x), h).
\label{eq:b45}
\end{equation}
The extension $\delta\in Der A$ is given by the formula
\begin{equation}
\delta(x+h)=\sigma x+\gamma(x)+\tau (h), ~\mbox{ for all}~~ x\in M, ~~ h\in H.
\label{eq:b46}
\end{equation}
\label{thm:t43}
\end{theorem}

\begin{proof} If there exists  a linear map $\gamma: M\rightarrow H$ satisfies Eq.(\ref{eq:b43}), Eq.(\ref{eq:b44}) and Eq.(\ref{eq:b45}), thanks to  Eq. (\ref{eq:b1}),
the linear map $\delta: A\rightarrow A$ defined by  Eq.(\ref{eq:b46}) satisfies that
for all $x_j\in M, j=1, 2, 3,$ $h, h_1, h_2\in H$

$\delta[x_1, x_2, x_3]_A=\delta([x_1, x_2, x_3]+\mu(x_1, x_2, x_3))$
\\
$=\sigma[x_1, x_2, x_3]+\gamma[x_1, x_2, x_3]+\tau\mu(x_1, x_2, x_3),$

$[\delta x_1, x_2, x_3]_A+[x_1, \delta x_2, x_3]_A+[x_1, x_2, \delta x_3]_A$
\\$=[\sigma x_1+\gamma(x_1), x_2, x_3]_A+[x_1, \sigma x_2+\gamma(x_2), x_3]_A+[x_1, x_2, \sigma x_3+\gamma(x_3)]_A$
\\$=[\sigma x_1, x_2, x_3]+\mu(\sigma x_1, x_2, x_3)+\rho(x_2, x_3)\gamma(x_1)+[x_1, \sigma x_2, x_3]$
\\
$+\mu(x_1, \sigma x_2, x_3)$$-\rho(x_1, x_3)\gamma(x_2)+[x_1, x_2, \sigma x_3]$
\\
$+\mu(x_1, x_2, \sigma x_3)+\rho(x_1, x_2)\gamma(x_3).$

By Eq.(\ref{eq:b43}),

$\delta[x_1, x_2, x_3]_A=[\delta x_1, x_2, x_3]_A+[x_1, \delta x_2, x_3]_A+[x_1, x_2, \delta x_3]_A$.

$\delta[x_1, x_2, h]_A=\tau(\rho(x_1, x_2)h),$ and

$[\delta x_1, x_2, h]_A+[x_1, \delta x_2, h]_A+[x_1, x_2, \delta h]_A$
\\$=[\sigma x_1+\gamma(x_1), x_2, h]_A+[x_1, \sigma x_2+\gamma(x_2), h]_A+[x_1, x_2, \tau h]_A$
\\$=\rho(\sigma x_1, x_2)h-\beta(x_2, \gamma(x_1))h+\rho(x_1, \sigma x_2)h+\beta(x_1, \gamma(x_2))h+\rho(x_1, x_2)\tau h.$

By Eq.(\ref{eq:b44}),

$\delta[x_1, x_2, h]_A=[\delta x_1, x_2, h]_A+[x_1, \delta x_2, h]_A+[x_1, x_2, \delta h]_A$.

$\delta[x, h_1, h_2]_A=\delta\beta(x, h_1)h_2=\tau\beta(x, h_1)h_2,$
and

$[\delta x, h_1, h_2]_A+[x, \delta h_1, h_2]_A+[x, h_1, \delta h_2]_A$
\\$=\beta(\sigma x, h_1)h_2+[\gamma(x), h_1, h_2]+[x, \tau h_1, h_2]+[x, h_1, \tau h_2]$
\\$=\beta(\sigma x, h_1)h_2+[\gamma(x), h_1, h_2]+\beta(x, \tau h_1)h_2+\beta(x, h_1)\tau h_2.$

By Eq.(\ref{eq:b45}),

 $\delta[x, h_1, h_2]_A=\delta\beta(x, h_1)h_2$
 \\$=[\delta x, h_1, h_2]_A+[x, \delta h_1, h_2]_A+[x, h_1, \delta h_2]_A$.
\\It follows that $\delta$ is a derivation of $A$ and the diagram (\ref{eq:b41}) is commutative.

Conversely, by Theorem \ref{cor:t32}, and the diagram (\ref{eq:b41}), for all $x\in A, h\in H$,
$$P(\delta (x))=\sigma P(x)=\sigma x, ~\delta(i(h))=\delta(h)= i(\tau(h))=\tau(h),$$
and   $\delta(x)-\sigma (x)\in H, \delta(h)=\tau (h).$

Define linear map $\gamma: A= M\dot+H \rightarrow H$, by $\gamma(x+h)=\delta(x)-\sigma(x),$ for all $x\in M$ and $h\in H$.

Since $\delta\in Der(A),$  for all $x_j\in M, j=1, 2, 3,$

$\delta[x_1, x_2, x_3]_A=\delta([x_1, x_2, x_3]+\mu(x_1, x_2, x_3))$
\\
$=\sigma[x_1, x_2, x_3]+\gamma[x_1, x_2, x_3]+\tau\mu(x_1, x_2, x_3)$
\\
$=[\delta x_1, x_2, x_3]_A+[x_1, \delta x_2, x_3]_A+[x_1, x_2, \delta x_3]_A$
\\$=[\sigma x_1+\gamma(x_1), x_2, x_3]_A+[x_1, \sigma x_2+\gamma(x_2), x_3]_A+[x_1, x_2, \sigma x_3+\gamma(x_3)]_A$
\\$=[\sigma x_1, x_2, x_3]+\mu(\sigma x_1, x_2, x_3)+\rho(x_2, x_3)\gamma(x_1)+[x_1, \sigma x_2, x_3]+\mu(x_1, \sigma x_2, x_3)$

\vspace{1mm} $-\rho(x_1, x_3)\gamma(x_2)+[x_1, x_2, \sigma x_3]+\mu(x_1, x_2, \sigma x_3)+\rho(x_1, x_2)\gamma(x_3),$
\\it follows that Eq.(\ref{eq:b43}) holds.
For all $x_1, x_2\in M$ and $ h\in H,$ thanks to  Eq. (\ref{eq:b1}),

$\delta[x_1, x_2, h]_A=\tau(\rho(x_1, x_2)h)$
\\$=[\delta x_1, x_2, h]_A+[x_1, \delta x_2, h]_A+[x_1, x_2, \delta h]_A$
\\$=[\sigma x_1+\gamma(x_1), x_2, h]_A+[x_1, \sigma x_2+\gamma(x_2), h]_A+[x_1, x_2, \tau h]_A$
\\$=\rho(\sigma x_1, x_2)h-\beta(x_2, \gamma(x_1))h+\rho(x_1, \sigma x_2)h+\beta(x_1, \gamma(x_2))h+\rho(x_1, x_2)\tau h.$

We obtain Eq.(\ref{eq:b44}).
Moreover, for all $x\in M, h_1, h_2\in H,$

$\delta[x, h_1, h_2]_A=\delta\beta(x, h_1)h_2=\tau\beta(x, h_1)h_2$
\\$=[\delta x, h_1, h_2]_A+[x, \delta h_1, h_2]_A+[x, h_1, \delta h_2]_A$
\\$=\beta(\sigma x, h_1)h_2+[\gamma(x), h_1, h_2]+[x, \tau h_1, h_2]+[x, h_1, \tau h_2]$
\\$=\beta(\sigma x, h_1)h_2+[\gamma(x), h_1, h_2]+\beta(x, \tau h_1)h_2+\beta(x, h_1)\tau h_2.$
\\It follows Eq.(\ref{eq:b45}).  The proof is complete.
\end{proof}

\begin{corollary} Let $A=M\dot+ H$ be a $(\mu, \rho)$-extension $3$-Lie algebra of $H$ by $M$, and $(H, \rho)$ be an $M$-module and $(\sigma, \tau)\in Der(M)\times Der(H)$.
 Then the pair $(\sigma, \tau)$ is extendable if and only if  Eq.(\ref{eq:b43}) holds and there exists  a linear map $\gamma: M\rightarrow H$ satisfying $\gamma(M)\subseteq Z(H)$ and
\begin{equation}
\tau\rho(x_1, x_2)-\rho(x_1, x_2)\tau=\rho(\sigma x_1, x_2)+\rho(x_1, \sigma x_2), \forall x_1, x_2\in M, h\in H.
\label{eq:b47}
\end{equation}

\label{cor:c41}
\end{corollary}

\begin{proof} If there exists a linear map $\gamma: M\rightarrow H$ satisfying $\gamma(M)\subseteq Z(H)$ and Eq.(\ref{eq:b43}) and Eq.(\ref{eq:b47}), it is easy to show that $\delta$  defined by  Eq.(\ref{eq:b46})  is a derivation of $A$, and the diagram  Eq.(\ref{eq:b41}) is commutate.

Conversely, by the  diagram  (\ref{eq:b41}), $P(\delta (x))=\sigma P(x)=\sigma x, ~ i(\tau h)=\delta(i(h))$. Then we obtain a linear map $\gamma: M\rightarrow H$, $\gamma(x)= \delta(x)-\sigma x$ which satisfies that,  for all ~$x, x_1, x_2, x_3\in M, ~h_1, h_2\in H,$

$[\gamma (x), h_1, h_2]=[\sigma x+\gamma (x), h_1, h_2]$
\\$=[\delta x, h_1, h_2]_A+[x, \delta h_1, h_2]_A+[x, h_1, \delta h_2]_A=\delta[x, h_1, h_2]_A$
\\$=0.$
\\ It follows $\gamma(M)\subseteq Z(H)$.

 Since $\delta\in Der(A)$,

$\delta[x_1, x_2, x_3]_A=\delta([x_1, x_2, x_3]+\mu(x_1, x_2, x_3))$
\\$=\sigma[x_1, x_2, x_3]+\gamma[x_1, x_2, x_3]+\tau\mu(x_1, x_2, x_3)$
\\
$=[\delta x_1, x_2, x_3]_A+[x_1, \delta x_2, x_3]_A+[x_1, x_2, \delta x_3]_A$
\\
$=[\sigma x_1+\gamma(x_1), x_2, x_3]_A+[x_1, \sigma x_2+\gamma(x_2), x_3]_A+[x_1, x_2, \sigma x_3+\gamma(x_3)]_A$
\\$=[\sigma x_1, x_2, x_3]+\mu(\sigma x_1, x_2, x_3)+\rho(x_2, x_3)\gamma(x_1)+[x_1, \sigma x_2, x_3]$
\\
$+\mu(x_1, \sigma x_2, x_3)+\rho(x_3, x_1)\gamma(x_2)+[x_1, x_2, \sigma x_3]$
\\$+\mu(x_1, x_2, \sigma x_3)+\rho(x_1, x_2)\gamma(x_3),$
\\ it follows that  Eq(\ref{eq:b43}) holds. From

$\delta[x_1, x_2, h]_A=\tau(\rho(x_1, x_2)h)$=
$[\delta x_1, x_2, h]_A+[x_1, \delta x_2, h]_A+[x_1, x_2, \delta h]_A$
\\
$=[\sigma x_1+\gamma(x_1), x_2, h]_A+[x_1, \sigma x_2+\gamma(x_2), h]_A+[x_1, x_2, \tau h]_A$
\\
$=\rho(\sigma x_1, x_2)h+\rho(x_1, \sigma x_2)h+\rho(x_1, x_2)\tau h.$

Hence Eq.(\ref{eq:b47}) holds.  The proof is complete.
\end{proof}

\begin{lemma} Let $A$ be a $3$-Lie algebra,  $\delta\in End(A)$. Define $f_{\delta}: A\rightarrow A^3$ by the formula
\begin{equation}
 f_{\delta}(x)=(\delta x, x, x), \forall x\in A.
\label{eq:e48}
\end{equation} Then  $\delta\in Der(A)$ if and only if $f_{\delta}$ is a $3$-Lie  homomorphism.
\label{lem:le41}
\end{lemma}

\begin{proof} By Eq.(\ref{eq:e48}) and Eq.(\ref{eq:b42}),

$f_{\delta}[x, y, z]=(\delta[x, y, z], [x, y, z], [x, y, z])$,

$[f_{\delta}(x), f_{\delta}(y), f_{\delta}(z)]=[(\delta x, x, x), (\delta y, y, y), (\delta z, z, z)]$
\\$=([\delta x, y, z]+[\delta y, z, x]$ $+[\delta z, x, y],$ $ [x, y, z], [x, y, z]),$ for all $x, y, z\in A.$ The result follows. \end{proof}

Since  a pair of derivations $(\sigma, \tau)\in Der(M)\times Der(H)$ corresponds to the pair of $3$-Lie  homomorphisms $(f_{\sigma}, f_{\tau})$, where $f_{\sigma}: M\rightarrow  M^3$ and $f_{\tau}: H\rightarrow H^3$, this allows us to translate questions about extensions of derivations into questions about extensions of $3$-Lie homomorphisms.

We first prove the following lemma.

\begin{lemma} Let $H$ and $M$ be $3$-Lie algebras, $A=M\dot+H$ be a $(\mu, \rho, \beta)$-extension $3$-Lie algebra of $H$ by $M$.
Then the sequence of $3$-Lie algebras

\begin{equation}
   \begin{CD}
        0   @>>>    H^3    @>i\otimes i\otimes i>>       A^3       @>p\otimes p\otimes p>>   M^3      @>>>    0.
          \end{CD}
  \label{eq:e49}
\end{equation}
 is exact, where  $i\times i\times i$ is identity on $H^3$, and $p\times p\times p$ is the projection.
\label{lem:le42}
\end{lemma}

\begin{proof} The result follows from Eqs. (\ref{eq:b1}) and (\ref{eq:b42}) and a direct computation.\end{proof}

\begin{lemma} Let $H$ and $M$ be $3$-Lie algebras, $A=M\dot+H$ be a $(\mu, \rho, \beta)$-extension $3$-Lie algebra of $H$ by $M$, and $(\sigma, \tau)\in Der(M)\otimes Der(H)$, $f_{\sigma}: M\rightarrow M^3$ and $f_{\tau}: H\rightarrow H^3$ be defined by Eq.(\ref{eq:e48}). Then there exists  a $3$-Lie  homomorphism $g: A\rightarrow A^3$ such that the following diagram commutes

\begin{equation}
 \begin{CD}
        0   @>>>     H    @>i>>       A       @>p>>   M       @>>>    0   \\
        @.              @Vf_{\tau} VV           @Vg VV      @Vf_{\sigma} VV        \\
        0   @>>>     H^3    @>i\times i\times i>>       A^3      @>P\times P\times P>>   M^3     @>>>    0
    \end{CD}
 \label{eq:e50}
\end{equation}
\\
if and only if there exists linear map $\gamma=(\gamma_{1}, \gamma_2,\gamma_3): M\rightarrow H^3,$ such that  for all $x\in M$ and $h\in H$,

 $g(x+h)=(\gamma_{1}(x)+\sigma(x)+\tau(h), \gamma_{2}(x)+x+h, \gamma_{3}(x)+x+h)$, and
\begin{equation}
\gamma_{j}(M)\subseteq Z(H), ~ j= 2, 3,
\label{eq:e410}
\end{equation}

\begin{equation}
\beta(\gamma_{j}(x), y)+\beta(x, \gamma_{j}(y))=0,  j=2, 3,
\label{eq:e411}
\end{equation}

\begin{equation}
\gamma_{j}[x, y, z]=\rho(x, y)\gamma_{j}(z)+\rho(z, x)\gamma_{j}(y)+\rho(y, z)\gamma_{j}(x), j=2, 3,
\label{eq:e412}
\end{equation}

\begin{equation}
[\tau, \rho(x, y)]=\rho(\sigma x, y)+\rho(x, \sigma y)+\beta(\gamma_{1}(x), y)
\label{eq:e413}
\end{equation}
\hspace{2cm}$+\beta(x, \gamma_{1}(y))-\beta(\gamma_{2}(\sigma x), y)-\beta(x, \gamma_{2}(\sigma y)),$
\begin{equation}
[\tau, \beta(x, h)]=\beta(\sigma x, h)+\beta(x, \tau h)+ad(\gamma_{1}(x), h), \forall x\in M, h\in H,
\label{eq:e414}
\end{equation}

\begin{equation}
\tau\mu(x, y, z)+\gamma_{1}[x, y, z]-\gamma_{2}\sigma[x, y, z]
\label{eq:e415}
\end{equation}

\hspace{2cm}$=\mu(\sigma x, y, z)+\mu(x, \sigma y, z)+\mu(x, y, \sigma z)+\rho(x, y)\gamma_{1}(z)$

$
+\rho(z, x)\gamma_{1}(y)+\rho(y, z)\gamma_{1}(x)-\rho(x, y)\gamma_{2}(\sigma z)-\rho(z, x)\gamma_{2}(\sigma y)-\rho(y, z)\gamma_{2}(\sigma x).$

\label{lem:le43}
\end{lemma}

\begin{proof} Let  $g: A\rightarrow A^3$ be a $3$-Lie homomorphism, and denote  $g=(g_1, g_2, g_3)$,  $g_j: A\rightarrow A, j=1, 2, 3$. Then  for all $x, y, z, $ $x_j\in M, h, h_j\in H, j=1, 2, 3$,

\vspace{4mm} $g[x_1+h_1, x_2+h_2, x_3+h_3]_A$
\\
$=(g_{1}[x_1+h_1, x_2+h_2, x_3+h_3]_A, g_{2}[x_1+h_1, x_2+h_2, x_3+h_3]_A, g_{3}[x_1+h_1, x_2$
\\$+h_2, x_3+h_3]_A)$
\\
$=[g(x_1+h_1), g(x_2+h_2), g(x_3+h_3)]_A$
\\
$=([g_{1}(x_1+h_1), g_{2}(x_2+h_2), g_{2}(x_3+h_3)]_A+[g_{1}(x_3+h_3), g_{2}(x_1+h_1), g_{2}(x_2+h_2)]_A$
\\
$+ [g_{1}(x_2+h_2), g_{2}(x_3+h_3), g_{2}(x_1+h_1)]_A, [g_{2}(x_1+h_1), g_{2}(x_2+h_2), g_{2}(x_3$
\\$+h_3)]_A, [g_{3}(x_1+h_1), g_{3}(x_2$
$+h_2), g_{3}(x_3+h_3)]_A).$
We show that

$g_{2}[x_1+h_1, x_2+h_2, x_3+h_3]_A=[g_{2}(x_1+h_1), g_{2}(x_2+h_2), g_{2}(x_3+h_3)]_A,$ and

$g_{3}[x_1+h_1, x_2+h_2, x_3+h_3]_A=[g_{3}(x_1+h_1), g_{3}(x_2+h_2), g_{3}(x_3+h_3)]_A,$ that is,
$g_2, g_3$ are $3$-Lie homomorphisms.

 By the commutativity of diagram (\ref{eq:e49}),  $g|_{H}=f_{\tau}$, and $g_{2}|_{H}, g_{3}|_{H}$ are identities. Since for all $x\in M$,

 $f_{\sigma}(Px)=f_{\sigma}(x)=(\sigma x, x, x)$
 \\$=(P\times P\times P)(g(x))=(Pg_{1}(x), Pg_{2}(x), Pg_{3}(x))$,
 \\we get a linear map
 $\gamma=(\gamma_{1}, \gamma_2, \gamma_3):  M\rightarrow H^3$, defined by
  $$\gamma_{1}(x)=g_{1}(x)-\sigma(x),~~ \gamma_{2}(x)=g_{2}(x)-x, ~~\gamma_{3}(x)=g_{3}(x)-x.$$

From

$g_2[x, h_1, h_2]_A=\beta(x, h_1)h_2$$=[x+\gamma_{2}(x), h_1, h_2]_A$
\\  $=\beta(x, h_1)h_2 + [\gamma_{2}(x), h_1, h_2],$\\
we obtain  $[\gamma_{2}(x), h_1, h_2]=0,$ that is, $\gamma_{2}(M)\subseteq Z(H).$  Similarly, $\gamma_{3}(M)\subseteq Z(H),$
it shows that  Eq.(\ref{eq:e410}) holds.
Since

$g_{2}[x, y, h]_A=\rho(x, y)h$$=[x+\gamma_{2}(x), y+\gamma_{2}(y), h]_A$
\\$=\rho(x, y)h+\beta(\gamma_{2}(x), y)h+\beta(x, \gamma_{2}(y))h, $
\\we get $\beta(\gamma_{2}(x), y)+\beta(x, \gamma_{2}(y))=0.$ Similarly, $\beta(\gamma_{3}(x), y)+\beta(x, \gamma_{3}(y))=0.$ Therefore Eq.(\ref{eq:e411}) holds.
From

$g_{2}[x, y, z]_A=[x, y, z]+\gamma_{2}[x, y, z]+\mu(x, y, z)$
\\$=[x+\gamma_{2}(x), y+\gamma_{2}(y), z+\gamma_{2}(z)]_A$
\\$=[x, y, z]+\mu(x, y, z)+\rho(x, y)\gamma_{2}(z)+\rho(z, x)\gamma_{2}(y)+\rho(y, z)\gamma_{2}(x)$
\\$+[x, \gamma_{2}(y), \gamma_{2}(z)]+[\gamma_{2}(x), y, \gamma_{2}(z)]+[\gamma_{2}(x), \gamma_{2}(y), z],$
 \\we obtain

 $\gamma_{2}[x, y, z]=\rho(x, y)\gamma_{2}(z)+\rho(z, x)\gamma_{2}(y)+\rho(y, z)\gamma_{2}(x)$
 \\$+[x, \gamma_{2}(y), \gamma_{2}(z)]+[\gamma_{2}(x), y, \gamma_{2}(z)]+[\gamma_{2}(x), \gamma_{2}(y), z].$

Thanks to Eq.(\ref{eq:e411}),

$[x, \gamma_{2}(y), \gamma_{2}(z)]=-[\gamma_{2}(x), y, \gamma_{2}(z)]$
$=[\gamma_{2}(x), \gamma_{2}(y), z]$$=-[x, \gamma_{2}(y), \gamma_{2}(z)],$

 $[x, \gamma_{2}(y), \gamma_{2}(z)]=[\gamma_{2}(x), y, \gamma_{2}(z)]=[\gamma_{2}(x), \gamma_{2}(y), z]=0$, and

 $[M, \gamma_{2}(M), \gamma_{2}(M)]=0.$
Therefore,

$\gamma_{2}[x, y, z]=\rho(x, y)\gamma_{2}(z)+\rho(z, x)\gamma_{2}(y)+\rho(y, z)\gamma_{2}(x)$, \\that is, $\gamma_{2}\in Der(M, Z(H))$. Similarly,

$\gamma_{3}[x, y, z]=\rho(x, y)\gamma_{3}(z)+\rho(z, x)\gamma_{3}(y)+\rho(y, z)\gamma_{3}(x),$ \\and $ \gamma_{3}(x)\in Der(M, Z(H)).$ It follows that Eq.(\ref{eq:e412}) holds.

By $g[x, y, z]_A=g([x, y, z]+\mu(x, y, z))$
\\$=(g_{1}[x, y, z], g_2[x, y, z], g_3[x, y, z])+(\tau\mu(x, y, z), \mu(x, y, z), \mu(x, y, z))$
\\$=(\sigma[x, y, z]+\gamma_{1}[x, y, z]+\tau\mu(x, y, z), [x, y, z]+\gamma_{2}[x, y, z]$
\\$+\mu(x, y, z), [x, y, z]+\gamma_{3}[x, y, z]+\mu(x, y, z))$
   \\$=[(\sigma x+\gamma_{1}(x), x+\gamma_{2}(x), x+\gamma_{3}(x)), (\sigma y+\gamma_{1}(y), y+\gamma_{2}(y), y$\\
   $+\gamma_{3}(y)), (\sigma z+\gamma_{1}(z), z+\gamma_{2}(z), z+\gamma_{3}(z))]_A$
\\$=([\sigma x+\gamma_{1}(x), y+\gamma_{2}(y), z+\gamma_{2}(z)]_A+[\sigma z+\gamma_{1}(z), x+\gamma_{2}(x), y+\gamma_{2}(y)]_A$
\\$+[\sigma y+\gamma_{1}(y), z+\gamma_{2}(z), x+\gamma_{2}(x)]_A, [x+\gamma_{2}(x), y+\gamma_{2}(y),z+\gamma_{2}(z)]_A, [x+\gamma_{3}(x), y+\gamma_{3}(y), z+\gamma_{3}(z)]_A )$
\\$=([\sigma x, y, z]+\mu(\sigma x, y, z)+[\sigma x, y, \gamma_{2}(z)]+[\sigma x, \gamma_{2}(y), z]$
\\$+[\sigma x, \gamma_{2}(y), \gamma_{2}(z)]$$+[\gamma_{1}(x), y, z]$
$+[\gamma_{1}(x), y, \gamma_{2}(z)]$
\\$+[\gamma_{1}(x), \gamma_{2}y, z]+[\sigma z, x, y]+\mu(\sigma z, x, y)+[\sigma z, x, \gamma_{2}(y)]$
\\$+[\sigma z, \gamma_{2}(x), y]$
$+[\sigma z, \gamma_{2}(x), \gamma_{2}(y)]+[\gamma_{1}(z), x, y]$\\
$+[\gamma_{1}(z), x, \gamma_{2}(y)]$
$+[\gamma_{1}(z), \gamma_{2}(x), y]+[\sigma y, z, x]$
\\$+\mu(\sigma y, z, x)$
$+[\sigma y, z, \gamma_{2}(x)]+[\sigma y, \gamma_{2}(z), x]+[\sigma y, \gamma_{2}(z), \gamma_{2}(x)]$
\\$+[\gamma_{1}(y), z, x]+[\gamma_{1}(y), z, \gamma_{2}(x)]$
$+[\gamma_{1}(y), \gamma_{2}(z), x], [x, y, z]+\mu(x, y, z)$
\\$+[x, y, \gamma_{2}(z)]$
$+[x, \gamma_{2}(y), z]+[x, \gamma_{2}(y), \gamma_{2}(z)]$
$+[\gamma_{2}(x), y, z]+[\gamma_{2}(x), y, \gamma_{2}(z)]$
\\$+[\gamma_{2}(x), \gamma_{2}(y), z], [x, y, z]+\mu(x, y, z)+[x, y, \gamma_{3}(z)]$
\\$+[x, \gamma_{3}(y), z]+[x, \gamma_{3}(y), \gamma_{3}(z)]+[\gamma_{3}(x), y, z]$
\\$+[\gamma_{3}(x), y, \gamma_{3}(z)]+[\gamma_{3}(x), \gamma_{3}(y), z]).$

We obtain that

 $\gamma_{1}[x, y, z]+\tau\mu(x, y, z)$
\\$=[\gamma_{1}(x), y, z]+[\gamma_{1}(z), x, y]+[\gamma_{1}(y), z, x]+\mu(\sigma x, y, z)+\mu(\sigma z, x, y)$
\\$+\mu(\sigma y, z, x)+[\sigma x, y, \gamma_{2}(z)]+[\sigma x, \gamma_{2}(y), z]+[\sigma x, \gamma_{2}(y), \gamma_{2}(z)]+[\gamma_{1}(x), y, \gamma_{2}(z)]$
\\$+[\gamma_{1}(x), \gamma_{2}(y), z]+[\sigma z, x, \gamma_{2}(y)]+[\sigma z, \gamma_{2}(x), y]$
\\$+[\sigma z, \gamma_{2}(x), \gamma_{2}(y)]+[\gamma_{1}(z), x, \gamma_{2}(y)]+[\gamma_{1}(z), \gamma_{2}(x), y]$
\\$+[\sigma y, z, \gamma_{2}(x)]+[\sigma y, \gamma_{2}(z), x]+[\sigma y, \gamma_{2}(z), \gamma_{2}(x)]$
\\$+[\gamma_{1}(y), z, \gamma_{2}(x)]+[\gamma_{1}(y), \gamma_{2}(z), x].$

Thanks to Eq.(\ref{eq:e412}),

$[\sigma x, y, \gamma_{2}(z)]+[\sigma x, \gamma_{2}(y), z]=\gamma_{2}[\sigma x, y, z]-[\gamma_{2}(\sigma x), y, z],$

$[\sigma z, x, \gamma_{2}(y)]+[\sigma z, \gamma_{2}(x), y]$
$=\gamma_{2}[x, y, \sigma z]-[x, y, \gamma_{2}(\sigma z)],$
\\and

$[\sigma y, z, \gamma_{2}(x)]+[\sigma y, \gamma_{2}(z), x]=\gamma_{2}[x, \sigma y, z]-[x, \gamma_{2}(\sigma y), z],$
 \\we obtain

  $\gamma_{2}\sigma[x, y, z]=[\gamma_{2}(\sigma x), y, z]+[x, \gamma_{2}(\sigma y), z]+[x, y, \gamma_{2}(\sigma z)].$
\\From Eq.(\ref{eq:e411}) it follows that

 $\gamma_{1}[x, y, z]+\tau\mu(x, y, z)$
\\
$=[\gamma_{1}(x), y, z]+[\gamma_{1}(z), x, y]+[\gamma_{1}(y), z, x]+\mu(\sigma x, y, z)$
$+\mu(\sigma z, x, y)$
\\$+\mu(\sigma y, z, x)+\gamma_{2}\sigma[x, y, z]$
$-[\gamma_{2}(\sigma x), y, z]-[x, \gamma_{2}(\sigma y), z]-[x, y, \gamma_{2}(\sigma z)].$

Therefore, Eq.(\ref{eq:e415}) holds.

 Since

 $g[x, h_1, h_2]_A=(\tau[x, h_1, h_2], [x, h_1, h_2], [x, h_1, h_2])$
\\
 $=[(\sigma x+\gamma_{1}(x), x+\gamma_{2}(x), x+\gamma_{3}(x)), (\tau h_1, h_1, h_1), (\tau h_2, h_2, h_2)]_A$
\\
$=([\sigma x, h_1, h_2]+[\gamma_{1}(x), h_1, h_2]+[\tau h_2, x, h_1]+[\tau h_1, h_2, x], [x, h_1, h_2], [x, h_1, h_2]),$
\\
 we obtain

 $\tau[x, h_1, h_2]=[\sigma x, h_1, h_2]+[\gamma_{1}(x), h_1, h_2]+[\tau h_2, x, h_1]+[\tau h_1, h_2, x]$, it shows that  Eq.(\ref{eq:e414}) holds.

Thanks Eq.(\ref{eq:e411}) and Eq.(\ref{eq:e412}), and

 $g[x, y, h]_A=(\tau[x, y, h], [x, y, h], [x, y, h])$
\\
 $=[gx, gy, gh]_A=[(\sigma x+\gamma_{1}(x), x+\gamma_{2}(x), x+\gamma_{3}(x)), (\sigma y+\gamma_{1}(y), y$
 \\$+\gamma_{2}(y), y+\gamma_{3}(y)), (\tau h, h, h)]_A$
\\
$=([\sigma x, y, h]+[\sigma x, \gamma_{2}(y), h]+[\gamma_{1}(x), y, h]+[\tau h, x, y]+[\tau h, x, \gamma_{2}(y)]$
\\$+[\tau h, \gamma_{2}(x), y]$$+[\sigma y, h, x]+[\sigma y, h, \gamma_{2}(x)]+[\gamma_{1}(y), h, x], [x, y, h]$
\\$+[\gamma_{2}(x), y, h]+[x, \gamma_{2}(y), h], [x, y, h]+[\gamma_{3}(x), y, h]$
$+[x, \gamma_{3}(y), h]),$ \\ It follows  that Eq.(\ref{eq:e413}) holds.

Conversely, define $g=(g_1, g_2, g_3):$$ A\rightarrow A^3$, by for all $x\in M$ and $h\in H$,

$g_{1}(x+h)=\gamma_{1}(x)+\sigma(x)+\tau(h), g_2(x+h)=\gamma_{2}(x)+x+h, g_3(x+h)$
\\$=\gamma_{3}(x)+x+h.$
\\By a direct computation according to  Eqs.(\ref{eq:e410}) -  Eqs.(\ref{eq:e415}), $g$ is a $3$-Lie homomorphism and the diagram  (\ref{eq:e49}) is commutate. The proof is complete.
\end{proof}

\begin{corollary} Let $H, M$ be $3$-Lie algebras,  $\rho: M \wedge M\rightarrow Der(H)$ be a representation, $A=M\dot+H$ be a $(\mu, \rho)$-extension $3$-Lie algebra of $H$ by $M$. Then for
$(\sigma, \tau)\in Der(M)\times Der(H)$, there exists a $3$-Lie homomorphism $g: A\rightarrow A^3$  such that  diagram (\ref{eq:e49}) commutes if and only if there exists $\gamma=(\gamma_1, \gamma_2, \gamma_3): M\rightarrow H^3$ satisfying Eq.(\ref{eq:e410}), Eq.(\ref{eq:e412}), Eq.(\ref{eq:e415}) and Eq.(\ref{eq:b47}).
\label{cor:c42}
\end{corollary}

\begin{proof}
The result follows from Lemma \ref{lem:le43} directly (the case $\beta=0$).

\end{proof}

From Lemma \ref{lem:le42} and Eq.(\ref{eq:e412}), $\gamma_{j}\in Der(M, Z(H)), j=2, 3$, so we can extend $\gamma_{j}, j=2, 3$ to a derivation of $A$ into $Z(H)$ by setting $$\gamma_{j}(H)=0, j=2, 3.$$ Then we obtain the following result directly.

\begin{corollary} Let $H, M$ be $3$-Lie algebras, $A=M\dot+H$ be a $(\mu, \rho, \beta)$-exten\-sion $3$-Lie algebra of $H$ by $M$. If $\gamma\in Der(A, Z(H))$ satisfies Eqs.(\ref{eq:e410})-Eq.(\ref{eq:e415}) in Lemma \ref{lem:le43} and $\gamma(H)=0$,  then $g_{\gamma}$ is a $3$-Lie homomorphism,
where $g_{\gamma}: A^3\rightarrow A^3, g_{\gamma}(x, y, z)=(x+\gamma(x), y+\gamma(y), z+\gamma(z))$.
\label{cor:c43}
\end{corollary}

\begin{theorem} Let $H, M$ be $3$-Lie algebras, $A=M\dot+H$ be a $(\mu, \rho, \beta)$-exten\-sion $3$-Lie algebra of $H$ by $M$. A pair $(\sigma, \tau)\in Der(M)\times Der(H)$ can be extended  to a derivation of $A$ if and only if the pair of $3$-Lie homomorphisms $(f_{\sigma}, f_{\tau})$ can be extended to a $3$-Lie  homomorphism  from $A$ to $A^3$.
\label{thm:t45}
\end{theorem}

\begin{proof} If $g$ is the extension of $(f_{\sigma}, f_{\tau})$, thanks to Lemma \ref{lem:le43} and Corollary \ref{cor:c43}, we know for all $x\in M, h, h_1, h_2\in H$,

\vspace{4mm} $g_{-\gamma}\circ g(x+h)=g_{-\gamma}(g_{1}(x+h), g_{2}(x+h), g_{3}(x+h))$

\noindent$=g_{-\gamma}(\sigma x+\gamma_{1}(x)+\tau h, x+\gamma_{2}(x)+h, x+\gamma_{3}(x)+h)$

\noindent$=(\sigma x+\gamma_{1}(x)+\tau h-\gamma(\sigma x), x+\gamma_{2}(x)+h-\gamma(x), x+\gamma_{3}(x)+h-\gamma(x))$

\noindent$=(\sigma x+\gamma_{1}(x)+\tau h-\gamma(\sigma x), x+h, x+h).$

Define linear map: $\delta: A\rightarrow A$ by
$$\delta (x)=\sigma x+\gamma_{1}(x)-\gamma(\sigma x), \delta (h)=\tau h, \forall x\in M, h\in H.$$

Then for all $x, y, z\in M, h_{k}\in H, k=1, 2$, we have

 $\delta [x, h_1, h_2]_{A}=\tau[x, h_1, h_2]$, $\delta [x, y, h]=\tau[x, y, h],$

 $\delta [x, y, z]_{A}=\sigma[x, y, z]+\gamma_{1}[x, y, z]-\gamma(\sigma[x, y, z])+\tau\mu(x, y, z).$

$[\delta x, h_1, h_2]_{A}+[x, \delta h_1, h_2]_{A}+[x, h_1, \delta h_2]_{A}$

\noindent$=[\sigma x, h_1, h_2]+[x, \tau h_1, h_2]+[x, h_1, \tau h_2]+[\gamma_{1}(x), h_1, h_2].$

Due to $[\tau, \beta(x, h)]=\beta(\sigma x, h)+\beta(x, \tau h)+ad(\gamma_{1}(x), h)$, we have

$[\delta x, y, h]+[x, \delta y, h]+[x, y, \delta h]$
\\$=[\sigma x, y, h]+[\gamma_{1}x, y, h]-[\gamma(\sigma x), y, h]+[x, \sigma y, h]$
\\$+[x, \gamma_{1}(y), h]-[x, \gamma(\sigma y), h]+[x, y, \tau h].$

Again from

$[\tau, \rho(x, y)]$
\\$=\rho(\sigma x, y)+\rho(x, \sigma y)+\beta(\gamma_{1}(x), y)+\beta(x, \gamma_{1}(y))-\beta(\gamma(\sigma x), y)-\beta(x, \gamma(\sigma y))$,

$[\delta x, y, z]_{A}+[x, \delta y, z]_{A}+[x, y, \delta z]_{A}$
\\$=[\sigma x+\gamma_{1}(x)-\gamma(\sigma x), y, z]_{A}+[x, \sigma y+\gamma_{1}(y)-\gamma(\sigma y), z]_{A}$
\\$+[x, y, \sigma z+\gamma_{1}(z)-\gamma(\sigma z)]_{A}$
\\$=[\sigma x, y, z]+\mu(\sigma x, y, z)+[\gamma_{1}(x), y, z]-[\gamma(\sigma x), y, z]+[x, \sigma y, z]+\mu(x, \sigma y, z)$
\\$+[x, \gamma_{1}(y), z]$$-[x, \gamma(\sigma y), z]+[x, y, \sigma z]+\mu(x, y, \sigma z)+[x, y, \gamma_{1}(z)]$
\\$-[x, y, \gamma(\sigma z)],$ and

$\tau\mu(x, y, z)+\gamma_{1}[x, y, z]-\gamma\sigma[x, y, z]$
\\$=\mu(\sigma x, y, z)+\mu(x, \sigma y, z)+\mu(x, y, \sigma z)+\rho(x, y)\gamma_{1}(z)+\rho(z, x)\gamma_{1}(y)$
\\$+\rho(y, z)\gamma_{1}(x)-\rho(x, y)\gamma(\sigma z)-\rho(z, x)\gamma(\sigma y)-\rho(y, z)\gamma(\sigma x)$, \\we show that $\delta$ is a derivation of $A$ and it is an extension of the pair $(\sigma, \tau)$.

Conversely, the result follows from Lemma \ref{lem:le41} directly. The proof is complete.

\end{proof}

\subsection*{Acknowledgements}
\noindent The first named author was supported in part by the Natural
Science Foundation (11371245) and the Natural
Science Foundation of Hebei Province (A2014201006).

\vspace{1cm}
\noindent Ruipu Bai, Yansha Gao, Zhenheng Li\\
College of Mathematics and Information  Science\\
Hebei University, Baoding 071002, China\\

\noindent 
 bairuipu@hbu.edu.cn \\
 gaoyansha111@163.com\\
 zhenhengl@yahoo.com


\begin{thebibliography}{HD}




\normalsize
\baselineskip=17pt


\bibitem[A]{ALMY} H. Awata, M. Li, D. Minic and T. Yoneya, On the quantization of Nambu brackets, {\it JHEP}  2, (2001), 013 (17pp).

\bibitem[BBW]{BBW} R. Bai, C. Bai and J. Wang, Realizations of 3-Lie algebras, {\it J. Math. Phys.,} {\bf 51}, (2010), 063505.

\bibitem[BLZ]{BLZ33} R. Bai, H. Liu, M. Zhang, 3-Lie Algebras Realized by Cubic Matrices, {Chin. Ann. Math.,} {\bf 35B}, (2014), 2: 261-270.

\bibitem[BW]{BW22} R. Bai, Y. Wu, Constructions of 3-Lie algebras, {\it  Linear and Multilinear Algebra}, {\bf 63}, (2015), 11: 2171-2186.

\bibitem[F]{F}  V.
 Filippov, $n-$Lie algebras,  {\it Sib. Mat.
           Zh.,} {\bf 26} (1985): 126-140


\bibitem[G]{G} A. Gustavsson, Algebraic structures on parallel
M2-branes,{\it  Nucl.Phys.B} {\bf 811} (2009): 66-76
	

\bibitem[HIM]{HIM} P. Ho, Y. Imamura, Y. Matsuo, $M2$ to $D2$ revisited, {\it JHEP},
 {\bf 003} (2008), 0807



\bibitem[L]{L} W. Ling, On the structure of $n-$Lie
                  algebras, Dissertation, University-GHS-Siegen, Siegn, 1993.



\bibitem[N]{N} Y. Nambu, Generalized Hamiltonian dynamics, {\it Phys. Pev. D} {\bf 7}, (1973):
                2405-2412.




\bibitem[Po1]{Po1} A. Pozhidaev, Simple quotient algebras and subalgebras of Jacobian algebras, {\it Sib. Math. J.,} {\bf 39} (1998), 3: 512-517.

\bibitem[Po2]{Po2} A. Pozhidaev, Monomial n-Lie algebras, {\it Algebra i Logika,} {\bf 37} (1998), 3: 181-192.



\bibitem[T]{T} L. Takhtajan, On foundation of the generalized Nambu mechanics,
{\it Comm. Math. Phys.} {\bf 160}, (1994): 295-315.



\bibitem[STS]{STS} M. Semenov-Tian-Shansky,
     What is a classical $r$-matrix,
   {\em Funct. Ana. Appl.}, {\bf 17} (1983): 259-272.




\end{thebibliography}
\end{document}